\documentclass[12pt,a4paper]{amsart}

 \usepackage{amssymb, amstext, amscd, amsmath, amsfonts, amsthm, amscd, color}

\usepackage{tikz,mathdots,enumerate}
\usepackage{pgfplots}% This uses tikz
\usepackage{url}

\usepackage{tikz,mathdots,enumerate}

\usepackage{graphicx, array, blindtext}
\usepackage{url}
\usepackage{tikz}
\usetikzlibrary{shapes.geometric, calc}
\usepackage{caption}
\usepackage{array}
\usepackage{epsfig}
\usepackage{eucal}
\usepackage{latexsym}
\usepackage{mathrsfs}
\usepackage{textcomp}
\usepackage{verbatim}

\newtheorem{thm}{Theorem}[section]

\newtheorem{cor}[thm]{Corollary}

\newtheorem{lemma}[thm]{Lemma}

\newtheorem{prop}[thm]{Proposition}

\numberwithin{equation}{section}

\theoremstyle{definition}
\newtheorem{rem}[thm]{Remark}
\newtheorem{example}[thm]{Example}
\newtheorem{definition}[thm]{Definition}
\newcommand{\bC}{{\mathbb{C}}}

\newcommand{\bN}{{\mathbb{N}}}
\newcommand{\bQ}{{\mathbb{Q}}}
\newcommand{\bR}{{\mathbb{R}}}
\newcommand{\bT}{{\mathbb{T}}}

\newcommand{\bZ}{{\mathbb{Z}}}

  \newcommand{\A}{{\mathcal{A}}}
  \newcommand{\B}{{\mathcal{B}}}
  \newcommand{\C}{{\mathcal{C}}}

  \newcommand{\F}{{\mathcal{F}}}
  \newcommand{\G}{{\mathcal{G}}}

\renewcommand{\L}{{\mathcal{L}}}
  \newcommand{\M}{{\mathcal{M}}}

\renewcommand{\P}{{\mathcal{P}}}
  \newcommand{\Q}{{\mathcal{Q}}}
  \newcommand{\R}{{\mathcal{R}}}
\renewcommand{\S}{{\mathcal{S}}}
  \newcommand{\T}{{\mathcal{T}}}
  \newcommand{\U}{{\mathcal{U}}}
  \newcommand{\V}{{\mathcal{V}}}

\newcommand{\ul}{\underline  }
\newcommand{\ol}{\overline  }

\usepackage{verbatim}
\usepackage{tikz}
\usetikzlibrary{calc}

\begin{document}

%\hfill {\today}

\title[Parallelogram frameworks and flexible quasicrystals]
{Parallelogram frameworks and flexible quasicrystals}

%\title[Flexible quasicrystals and zero modes]{Flexible quasicrystals and zero modes}

%\title[Parallelogram frameworks and the zero mode spectrum of a quasicrystal]{Parallelogram frameworks and the zero mode spectrum of a quasicrystal}

%\title[Parallelogram frameworks and a zero mode spectrum for quasicrystals]{Parallelogram frameworks and a zero mode spectrum for quasicrystals}

%\title[Localised motions and wave vector spectra for quasicrystal frameworks]
%{Localised motions and wave vector spectra for quasicrystal frameworks} 
%\title[Flexes in crystals and Penrose rhomb frameworks]{Localised flexes in crystals and Penrose rhomb frameworks}
%{Local motions of Penrose rhomb frameworks and crystals.}

%[Wave vector spectra for aperiodic bar-joint frameworks]{Wave vector spectra for aperiodic bar-joint frameworks}

\author[S. C. Power]{S. C. Power}

\address{Dept.\ Math.\ Stats.\\ Lancaster University\\
Lancaster LA1 4YF \\U.K. }

\email{s.power@lancaster.ac.uk}

%\begin{abstract}
%The structure of linearly localised first order motions is determined for parallelogram bar-joint framework in the plane. This motivates the definition of a zero mode wave vector spectrum, the slippage spectrum  ${\bf K}_{\rm slip}(\G,\ul{a})$, 
%of a  Delone bar-joint framework $\G$. This is a set of lines in the reciprocal space of a reference basis $\ul{a}$, and for a Penrose rhomb framework it consists of 5 lines through the origin. %through the origin. 
%Hyperplane localised modes in a crystallographic framework $\C$ are well-known to imply lines in the zero mode (RUM) spectrum of $\C$. It is shown that, conversely, lines of wave vectors imply the existence of hyperplane localised modes. This motivates some other forms of zero mode spectra for Delone bar-joint frameworks associated with linearly localised modes.
%\end{abstract}
\begin{abstract}
The first-order flex space of the bar-joint framework $\G_P$ of a parallelogram tiling $P$ is determined in terms of an explicit free basis. Applications are given to  braced parallelogram frameworks and to quasicrystal frameworks associated with multigrids in the sense of de Bruijn and Beenker. 
%Also we define zero modes and a 
In particular we characterise rigid bracing patterns, identify quasicrystal frameworks with finite dimensional flex spaces, and define 
a zero mode spectrum.
\end{abstract}

\thanks{
Key words and phrases: braced parallelogram frameworks,  multigrid quasicrystals, zero modes.
%\\
%\hfill {\color{black}\today}
}

\maketitle

%\tableofcontents

\section{Introduction}

A characterisation is obtained for the infinitesimal flex space of the bar-joint framework $\G_P$ of a parallelogram tiling $P$ of the plane.
Of particular interest are the quasicrystal frameworks arising from substitution systems or from tilings that are dual to multigrids.
%, and in Section 2 we give the details of this construction. 
The Penrose rhomb tilings and the Ammann-Beenker octagonal tilings given in Section \ref{ss:multigridQCs}  are perhaps the most well-known but there are many other striking tilings, with $n$-fold symmetries, for arbitrary $n\geq 4$. See, for example, Baake and Grimm  \cite{baa-gri} and the online resource \cite{fre-et-al}. 
%Further parallelogram tilings associated with substitution systems may be readily viewed at the Encyclopedia of Tilings \cite{fre-et-al}. 
% See for example  D. Frettl\"oh, E. Harriss and F. G\"ahler \cite{fre-et-al}.
The determination of infinitesimal flexes 
is obtained in terms of a countable {free basis}, 
in the sense of Badri, Kitson and Power \cite{bad-kit-pow-2}, for the vector space of {all} infinitesimal flexes, 
%in the sense of Badri, Kitson and Power \cite{bad-kit-pow-2},    
where the vector fields of the basis are explicit and are related to the infinite ribbon structure of the tiling. Such an identification generally gives diverse information on flexibility and rigidity and this is the case here. 

In Theorem \ref{t:shearbasis} we characterise the  bracings of infinite parallelogram frameworks which ensure first-order rigidity. The necessary and sufficient condition, that the braces graph be connected and spanning, is the same requirement as the  well-known Bolker-Crapo condition \cite{bol-cra-siam}, \cite{pow-nonEuc} for the rigidity of finite braced rectangular grids.
Finite braced Penrose frameworks and parallelogram frameworks have been considered by Wester \cite{wes} and Francis and Duarte \cite{dua-fra} who show the sufficiency of the condition.
Other proofs of necessity and sufficiency in the finite case are due to Nagy Kem \cite{nag}, through a characterisation of associated tensegrity frameworks, and to Graseggar and Legersk\'y \cite{gra-leg} who employ elegant combinatorial colouring arguments. These methods  are different from our direct approach  exploiting the geometry of parallelogram ribbons.

Bar-joint frameworks and their forms of flexibility provide mathematical models in materials science for the flexibility and rigidity of network materials. This can be seen in the analysis of floppy modes (infinitesimal flexes) for Penrose rhomb quasicrystals and zero modes (floppy modes with wave vectors) for crystals. For some recent examples see %\cite{dov-exotic}, 
Rocklin et al \cite{roc-et-al}, Stenull and Lubensky  \cite{ste-lub}, and Zhou et al \cite{zho-et-al}. 
In this direction we show that general parallelogram frameworks do not admit floppy modes that tend to zero at infinity (Corollary \ref{c:novanishingflexes}). Also,
in Section \ref{s:multigridG} we consider regular  multigrid parallelogram tilings and associated quasicrystal frameworks which are obtained by bracing patterns. We define \emph{checkered quasicrystals} in this way and determine the finite dimension of their space of floppy modes (Theorem \ref{t:fintelyflexible}).

For multigrid parallelogram frameworks $\G_P$ we also characterise linearly localised floppy modes 
%Theorem \ref{t:bandltd_for_P} and 
(Theorem \ref{t:FiguresformultigridGP}). The set of lines through the origin for such flexes is shown to be equal to the \emph{ribbon figure}  RF$(P)$ of $P$, a finite set of lines representing the directions of the ribbons, and these directions are determined explicitly in Proposition \ref{p:ribbondirection}.

For a crystallographic material the zero modes (rigid unit modes) and zero mode spectrum (RUM spectrum) are determined, in accordance with Bloch's theorem,  by a finite geometric data set, for a repeating block of bars and joints, and by phase variations over translated blocks \cite{dov-exotic}, \cite{dov-2019}, \cite{pow-poly}, \cite{weg}. Despite this apparent limitation zero modes capture many aspects of global first order motion. %of $\C$.  
%This would amount to a quasicrystallographic analogue of the role of the 
Indeed, the RUM spectrum can be viewed as a generalised Bohr spectrum for which the associated  modes have dense linear span in the space of almost periodic first-order motions 
%(See Badri, Kitson and Power 
\cite{bad-kit-pow-1}. It is of interest then to formulate analogues, 
in the case of quasicrystals.
%quasicrystallographic bar-joint frameworks. 
We discuss some aspects of this in the final section and 
%in Definitions \ref{d:qc_phasefield} and  \ref{d:qc_velocityfield} we 
define zero modes and a zero mode spectrum for some quasicrystal bar-joint frameworks associated with multigrid parallelogram tilings.  

%MOVE: As in the crystallographic case these modes are determined as the points of rank degeneracy of a multi-variable matrix-valued 
%symbol function. However, the independent variables range in an $r$-torus, $\bT^r\subset \bC^r$, rather than a 2-torus, where $r$ is the number of component grids of an associated multigrid. 

%NOT HERE TEXT Accordingly, the generalised wave vectors for these modes are located in the reciprocal space of the lattice in higher dimensions that projects to the joints. We compute this spectrum for the checkered quasicrystals and raise some general questions.

%Such definitions may be apposite in applications and we pose some general questions. TO FINISH.. QUESTIONS .. etc

%In Section \ref{s:zeroenergyspec} we explore some forms of zero mode spectra for Delone bar-joint frameworks. In particular, the context of parallelogram frameworks motivates the definition of the slippage spectrum. The main idea is to consider phase-periodic velocity fields with respect to variable lattices. For a multigrid parallelogram tiling framework $\G_P$ it is equal to the reciprocal figure, relative to a reference basis, of the \emph{ribbon figure} RF$(P)$ of the tiling $P$. The ribbon figure is a set of lines in ambient space representing the directions of the ribbons of $P$, and an explicit formula for it is given in Proposition \ref{p:ribbondirection}. %

\section{Parallelogram frameworks}\label{s:penrose}
 We first discuss general parallelogram bar-joint frameworks $\G$ in $\bR^2$ and the nature of
their first-order motions. In particular we determine the structure of linearly localised motions and we characterise the bracing patterns for a parallelogram framework that make it rigid.
% and id and their directions in the case of regular multigrid parallelogram  frameworks, such as Penrose rhomb frameworks. 
 
A bar-joint framework $\G=(G,p)$ in $\bR^d$ is a finite or countable simple graph $G=(V,E)$ together with a placement $p:V \to \bR^d$ of its vertices.
A real \emph{infinitesimal flex} of  $\G$  is a vector field, or velocity field, $u:p(V) \to \bR^d$,  that satisfies the first-order flex condition for every bar. That is,
\[
\langle u(p(v))-u(p(w)), p(v)-p(w)\rangle = 0, \quad \mbox{ for } vw \in E.
\]
The space $\F(\G)$ of these fields is a subspace of the vector space $\V(\G)$ of all velocity fields, and $\G$ is said to be \emph{infinitesimally rigid}, or \emph{first-order rigid}, if it coincides with the space of \emph{rigid motion infinitesimal flexes}. For $d=2$ this is the 3-dimensional space spanned by a non-zero rotation infinitesimal flex and 2 linearly independent translation infinitesimal flexes.

A \emph{parallelogram tiling} $P$ is an embedded graph  in $\bR^2$ associated with a tiling of $\bR^2$ by contiguous nondegenerate parallelogram tiles. That is, nondisjoint tiles meet at a common vertex or common edge, and tiles have positive area.
We assume, moreover, to avoid nonstandard tilings, that the set of tiles is connected in the sense that any pair of tiles is connected by a path of contiguous tiles.
A \emph{parallelogram bar-joint framework} $\G_P$ in $\bR^2$ is a bar-joint framework whose embedded graph is a parallelogram tiling $P$.
In particular $\G_P$ possesses evident shearing motions associated with so-called ribbons of parallelograms.
A  \emph{ribbon}  of $P$ is a 2-way infinite path of adjacent tiles for which the shared edges have a common direction. We define a \emph{ribbon shear} of $\G_P$ to be an infinitesimal flex which is the zero vector on all joints on one side of the ribbon, including the ribbon boundary, and  assigns a common non-zero velocity vector to the other joints.

The following simple fact will be useful in the proof of Theorem \ref{t:shearbasis}.
See also Frettl\"oh and Harriss \cite{fre-har}.
\begin{center}
\begin{figure}[ht]
\centering
\includegraphics[width=6cm]{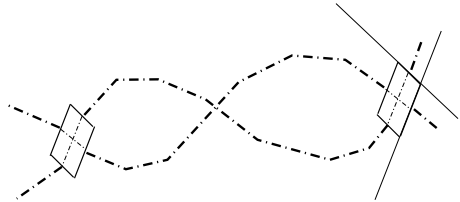}
\caption{A pair of impossible ribbons.}
\label{f:2ribbons}
\end{figure}
\end{center}
\begin{lemma}\label{l:2ribbons}
A pair of  ribbons of a parallelogram tiling share at most one tile.
\end{lemma}

\begin{proof}
Suppose that distinct ribbons $R_1$ and $R_2$ cross twice, at tiles $T_a, T_b$. These tiles are necessarily translationally equivalent. Consider the parts of these  ribbons that exit tile $T_a$ and  enter tile $T_b$. By the parallelogram structure of the ribbons these parts lie entirely in the same closed region bounded by lines $L_1, L_2$ through two edges of $T_b$.
Therefore, as Figure \ref{f:2ribbons} illustrates, these parts of the ribbons must cross at another tile. However we could have chosen $T_b$ to be the first tile encountered for a further crossing. This gives the desired contradiction.
\end{proof}

\subsection{The flex space of a parallelogram framework}
\label{ss:flexspaceparallelogram}
 Let $\G$ be a countable bar-joint framework.
A sequence $u^{(n)}$ of velocity fields in $\V(\G)$ \emph{tends to zero strictly} if for every joint $p_i$ of $\G$ the velocities $u^{(n)}(p_i)$ are nonzero for finitely many $n$. In this case any infinite sum $\sum_{n} \alpha_nu^{(n)}$, with real scalars $\alpha_n$, defines a velocity field. A \emph{free spanning set} for the space  $\F(\G)$ of infinitesimal flexes is a countable set in $\F(\G)$ that tends to zero strictly, in any enumeration, and is such that every infinitesimal flex  in $\F(\G)$ has such an infinite sum representation. When those sums are unique then we have the following definition of a free basis.
Free spanning sets and bases, and their symmetric variants for crystal frameworks, were examined in Badri, Kitson and Power \cite{bad-kit-pow-2}.

\begin{definition}\label{d:freebasis}
A \emph{free basis} for the infinitesimal flex space $\F(\G)$ of a countable bar-joint framework $\G$ in $\bR^d$ is a countable set $u^{(n)}$ in $\F(\G)$ which tends to zero strictly and is such that each infinitesimal flex has a unique  infinite sum representation $\sum_{n} \alpha_nu^{(n)}$.
\end{definition}

In the proof of the next theorem we use the following notation for a %Penrose rhomb 
parallelogram framework $\G$ with a distinguished  base joint $p_1$.
For each ribbon $\rho$ let $u^\rho$ be a nonzero ribbon shear for $\rho$ which assigns the zero velocity vector to $p_1$. Thus $u^\rho$ has zero velocities at all joints to one side of $\rho$, the side containing $p_1$, and is a fixed velocity vector at all other joints. Let $\S$ be the set of these ribbon shears. In particular $\S$ is countable and in any enumeration the $u^\rho$ tend to zero strictly.
Let $u^x, u^y$ be nonzero infinitesimal flexes for translation in the direction of the $x$-axis and the $y$-axis respectively.

\begin{thm}\label{t:shearbasis} Let $\G$ be a 
parallelogram bar-joint framework with base joint $p_1$ and an associated set $\S$ of ribbon shears. Then 
$\S \cup \{u^x, u^y\}$ is a free basis for $\F(\G)$.
\end{thm}

\begin{proof}
Let $z$ be an infinitesimal flex. 
Subtract a unique linear combination of $u^x, u^y$, say $w_1$, to obtain the infinitesimal flex $z_1=z-w_1$ with $z_1(p_1)=0$. Let $T_0$ be a tile that contains $p_1$ and let $\rho_a$ and $\rho_b$ be the ribbons through $T_0$.
We may add to $w_1$ a unique linear combination of the two corresponding ribbon shears to create $w_2$ and $z_2=z-w_2$ so that $z_2$ is zero on the 4 joints for $T_0$.
%$(p_i)=0,$ for $ 1\leq i \leq 4.$ 
We say 
%at this stage of the proof 
that the linear combination $w_2$ \emph{cancels  $z$ on the tile}  $T_0$, and that the ribbons ${\rho_a},{\rho_b}$ and the flexes $u^x,u^y,u^{\rho_a},u^{\rho_b}$, \emph{have been used}. 

\begin{center}
\begin{figure}[ht]
\centering
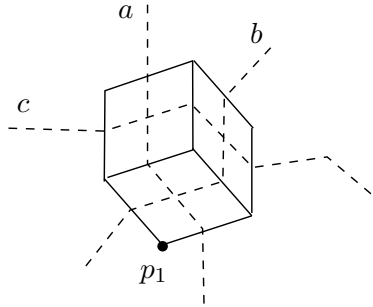
\caption{Three ribbons for the patch $T_0\cup T_1\cup T_2$. The lower tile, $T_0$, contains the vertex for the base joint $p_1$.}
\label{f:penrosepatch}
\end{figure}
\end{center}

Let $T_1$ be a tile that is contiguous to $T_0$, in the sense that it shares a tile edge with $T_0$, and lies on the unused ribbon $\rho_c$. Since $u^{\rho_c}$  vanishes on the joints for $T_0$ we may add an appropriate unique scalar multiple of $u^{\rho_c}$ to  $w_2$ to obtain $w_3$, which cancels $z$ on $T_0\cup T_1$. In this case we add $\rho_c$ to the set of used ribbons.

If there is a further tile $T_2$ contiguous to the patch 
$T_0\cup T_1$ which lies on 2 used ribbons then 3 of its joints lie at boundary vertices of this patch and so $z$ is also cancelled by $w_3$ on the larger connected patch $\P_0=T_0\cup T_1\cup T_2$. Figure \ref{f:penrosepatch} illustrates such a patch.

For the purposes of induction, assume that 
$\P$ is a simply connected patch of tiles, containing $T_0$, whose topological boundary is a simple closed curve.
Also, suppose that $\S_\P$ is a finite set of ribbon shears $u^\rho$ which have been used to obtain a unique linear combination $w_\P$ that cancels $z$ on $\P$. In particular each such ribbon $\rho$ includes at least one tile of the patch.  
We also assume that, like $\P_0$, the patch is maximal in the sense that there are no tiles outside $\P$ that have 3 or 4 vertices on the boundary of $\P$. 

Consider a tile $T$ not in $\P$ which shares an edge $e$ of $\P$. There is a unique ribbon $\tau$ which contains $T$ with $e$ lying in one of the two components of the boundary of $\tau$. We claim first that $\tau$ contains no tiles of $\P$. Suppose that this were so. Then, by the maximal property of $\P$ there is a nonempty patch $\Q$ of tiles between $\P$ and the boundary of $\tau$ containing $e$. See Figure \ref{f:reentrant}. Indeed, if $\Q$ is empty then there would have to be a final tile of $\tau$ outside $\P$ with 3 vertices in the boundary of $\P$. However, $\Q$ being nonempty is also not possible. To see this assume that $\Q$ has the least number of tiles for such a patch and that $T'$ is a tile in $\Q$ contiguous to $\P$. Then there is a ribbon though $T'$ which, by Lemma \ref{l:2ribbons}, must re-enters $\P$, and this contradicts the minimality of $\Q$.

Similarly, the set of tiles in $\tau$ that are contiguous to $\P$ must form a connected subset, or finite subribbon, of $\tau$. For otherwise there is a tile in $\Q$ which shares an edge $f$ of the boundary of $\P$, as in Figure \ref{f:reentrant}. 
This tile lies on a ribbon
$\tau'$ with $f$ in one of its boundary curves. By Lemma \ref{l:2ribbons} $\tau'$ must re-enter $\P$ and this is not possible.
%[OLD cannot share 2 tiles of $\tau$ and so there is a strictly smaller patch $\Q'$ between a boundary curve of $\tau'$ and $\P$, and so we obtain a contradiction.]
\begin{center}
\begin{figure}[ht]
\centering
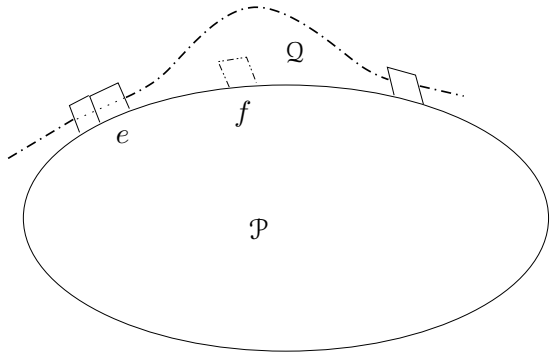
\caption{An impossible patch $\Q$ when $\tau$ is contiguous to $\P$.}
\label{f:reentrant}
\end{figure}
\end{center}

%\input{braceflex3.pstex_t}

%By the same argument it follows that the set of all tiles in $\tau$ which are contiguous to $\P$ form a connected subset of $\tau$. 
Now let $\P^+$ be the union of the tiles of the subribbon  with $\P$. Then there is a unique multiple of $u^\tau$ which when added to $w_\P$ cancels $z$ on $\P^+$. 
Let ${\P_1}$ be the, necessarily finite,  patch obtained from  $\P^+$ by successively adding tiles that share 3 vertices of the boundary.

In this way we complete an induction step. Moreover we may choose $T$ arbitrarily and so, since the tile-adjacency graph is connected, by our standing assumptions for $P$, it follows that a sequence of such choices is possible so that the union of the tiles of the resulting patch chain $\P_1, \P_2, \dots $ is the tiling $P$.

Note that the patch chain construction is independent of  the infinitesimal flex $z$, and so we obtain a set $\S_0$ of ribbon shears which, together with $u^x, u^y$ give a free basis for $\F(\G)$.

Finally we show that $\S_0=\S$. Let $u^\sigma \in \S\backslash \S_0$. Then,
since $u^\sigma(p_1)=0$ there is, by the construction above, a unique representation of the form  $u^\sigma = \sum_{\rho\in\S_0}\alpha_\rho u^\rho$ (where we write
${\rho\in\S_0}$ to denote ${u^\rho\in\S_0}$). 
Let $\tau\in \S_0$ be a ribbon that crosses $\sigma$.
Such ribbons exist since by the inductive construction every parallelogram lies on some ribbon of $\S_0$.
Let $\T$ be the 4-bar framework for the tile $T$ common to $\sigma$ and $\tau$. Since the restriction $u^{\sigma'}|_\T$ is an infinitesimal translation of $\T$, for every $\sigma' \neq \sigma, \tau$, it follows that $u^\sigma|_\T$ and $\alpha_{\tau}u^\tau|_\T$
differ by a translation. But
$u^\sigma$ and $u^\tau$ both have zero velocity vector on a joint of $\T$ and so we obtain the contradiction $u^\sigma|_\T = \alpha_{\tau}u^\tau|_\T$,  as desired.
\end{proof}

In view of our assumptions for a parallelogram tiling $P$ there are only finitely many vertices of $P$ in any disc $D(0,r)$ about the origin of finite radius $r$. We say that an infinitesimal flex $u$ of $\G_P$ \emph{tends to zero at infinity} if for each $\epsilon >0$ there is a radius $R>0$ such that $\|u(p_i)\|_2 <\epsilon$ for all joints $p_i$ located outside $D(0,R)$. Equivalently, for any enumeration of the joints, $p_1, p_2,\dots $, the sequence of velocities $u(p_1), u(p_2), \dots $ tends to zero.

\begin{cor}\label{c:novanishingflexes}
A parallelogram framework has no nonzero infinitesimal flex which tends to zero at infinity. In particular there are no nonzero infinitesimal flexes which are finitely supported or which are square-summable. 
\end{cor}

\begin{proof}
Suppose that the infinitesimal flex $u$ of a parallelogram framework $\G$ tends to zero at infinity. Consider a free basis as in the previous theorem, normalised so that the non zero velocity vector for each $u_\rho$ has unit norm, and consider a representation
\[
u = \alpha_xu^x+\alpha_yu^y + \sum_{\rho\in \S}\alpha_\rho u^\rho.
\]
Consider a ribbon $\rho$ for which the coefficient $\alpha_\rho$ is nonzero. Let $T_1, T_2, \dots $ be a sequence of tiles of $\rho$ which lie on the distinct ribbons $\tau_1, \tau_2, \dots $. By our assumption,
the restriction of $u$ to the associated subframeworks $\T_k$ tends to zero. It follows that 
%for $\epsilon>0$ 
there exists $k$ such that the restriction of
$\alpha_\rho u^\rho +\alpha_{\tau_k} u^{\tau_k}$ to the 4 joints of $\T_k$ have magnitude less than $1$. Indeed the restriction of any other basis flex $u^\sigma$ is a translation flex of $\T_k$ and so, in view of the series representation, the restriction of
$\alpha_\rho u^\rho +\alpha_{\tau_k} u^{\tau_k}$ can be made arbitrarily close to a translation. On the other hand these restrictions have a zero velocity at one of the 4 joints of $\T_k$ and so the conclusion follows.
We now have a contradiction since each of the ribbon shears $u^{\tau_k}$ vanishes on a joint of $\T_k$ where $u^\rho$ has unit norm. 
\end{proof}

%{\color{blue} floppy modes, boundaries, etc}
\begin{rem}
In Badri, Kitson and Power \cite{bad-kit-pow-2} it is shown by general infinite dimensional linear algebra that the infinitesimal flex space of every countable bar-joint framework $\G$ possesses free bases. However the proof is nonconstructive and for a specific framework the utility of a free basis or free spanning set comes from its  association with the geometry of $\G$.
%one needs more structural detail to determine whether there exists a free basis of localised flexes related to crystallographic symmetries and it is shown that such more structured free bases need not exist.{\color{magenta}Other shear bases .....to do}
\end{rem}

\subsection{When braced parallelogram frameworks are rigid}\label{ss:bracedparallelograms}
A \emph{braced parallelogram framework}, denoted $\G(\B)$, or $\G_P(\B)$, is a bar-joint framework obtained from a parallelogram bar-joint framework $\G_P$ by the addition of bars, known as braces, to parallelogram bar 4-cycles of $\G_P$. The set $\B$ denotes the set of these bars and it is assumed that at most one bar is added to any parallelogram bar 4-cycle. The associated \emph{braces graph}, denoted $G(\B)$, has vertex set labelled by the ribbons and edge set labelled by $\B$, or, equivalently, by the set of tiles $T$ whose bar 4-cycles are braced. Thus, the edge associated with $T$ is $(\rho, \sigma)$ where $\rho, \sigma$ are the ribbons through $T$. 

\begin{thm}\label{t:rigidbracing}
The braced parallelogram framework $\G(\B)$ is infinitesimally rigid if and only if $G(\B)$ is connected and spanning.
\end{thm}

\begin{proof}
We first show sufficiency. Let $u$ be an infinitesimal flex of $\G(\B)$. Subtracting a rigid motion infinitesimal flex we may assume that the restriction of $u$ to the 4 joints of a braced tile subframework $\T_1$, for the tile $T_1$, is zero. We show that $u=0$. Considering the previous theorem, where $p_1$ a joint of $T_1$, we have a representation
 $u= \sum_{\rho\in\S}\alpha_\rho u^\rho$. Moreover, from the inductive specification of the coefficients in this sum we have $\alpha_\tau = \alpha_\sigma=0$ for the ribbons $\tau, \sigma$ through $T_1$.  Let $\T_2$ be a braced tile subframework for a tile $T_2$ which lies on  $\tau$. Then, in view of the brace, the restriction $u|_{\T_2}$ is a rigid motion infinitesimal flex. On the other hand, if $\kappa$ is the other ribbon through $T_2$ then the restriction of $\sum_{\rho\in\S}\alpha_\rho u^\rho$ to $\T_2$ is equal to the restriction of $\alpha_\kappa u^\kappa$ plus an infinitesimal translation. It follows that $\alpha_\kappa=0.$ Since the braces graph is connected and spanning it follows that $\alpha_\sigma = 0$ for every ribbon $\sigma$, as required.
 
Note that if $G(\B)$ is not spanning then there is a ribbon $\rho$ with no associated braces and in this case $u^\rho$ is a nontrivial (non rigid motion) infinitesimal flex of $\G(\B)$. 

It remains to show that if $G(\B)$ is not connected then $\G(\B)$ has a nontrivial infinitesimal flex. Assume then that the braces graph is a disjoint union of 2 graphs, with an associated proper partition of the ribbons, $\R = \R_1\cup \R_2$, and the braces,
$\B = \B_1\cup \B_2$. Let $u$ be an infinitesimal flex with $u(p_1)=0$. By Theorem \ref{t:shearbasis} there exists a unique representation 
\[
u=u_1 +u_2, \quad 
%\sum_{\rho\in\R}\alpha_\rho u^\rho = 
u_1=\sum_{\rho\in\R_1}\alpha_\rho u^\rho, \quad 
u_2= \sum_{\rho\in\R_2}\alpha_\rho u^\rho.
\]
%\[
%u=\sum_{\rho\in\R}\alpha_\rho u^\rho = \sum_{\rho\in\R_1}\alpha_\rho u^\rho + \sum_{\rho\in\R_2}\alpha_\rho u^\rho = u_1 +u_2,
%\]
We first show that in view of the partitioning the velocity field $u_1$, and therefore $u_2$ also, is an infinitesimal flex of $\G(\B)$. Let  $\T$ be the bar 4-cycle for a tile $T$. If $T\in \B_1$ then every shearing flex $u^\rho$ for $\rho \in \R_2$ restricts to an infinitesimal translation of $\T$. Thus $u_2$ restricts to an infinitesimal translation of $\T$ and so $u_1=u-u_2$ restricted to $\T$ is an infinitesimal flex of $\T$. On the other hand if $T\in \B_2$ then, similarly, $u_1$ restricts to a infinitesimal translation of $\T$. It follows that $u_1$ is an infinitesimal flex of $\G(\B)$.

To obtain the desired contradiction assume further that $u$ is a nonzero infinitesimal rotation with $u(p_1)=0$.
Consider a pair $\rho_1\in \R_1, \rho_2\in \R_2$ with common tile $T$. Since the restriction of $u$ to the corresponding 4-joint framework $\T$ is an infinitesimal rotation it follows that both $\alpha_{\rho_1}$ and $\alpha_{\rho_2}$ are nonzero. This is because the restriction of every $u^\rho$ with $\rho \neq \rho_1, \rho_2$ is an infinitesimal translation. Thus the restriction $u_1|\T$ is not a rigid motion infinitesimal flex of $\T$ and so neither is $u_1$ itself.
\end{proof}

\begin{rem}
It should be clear from the proofs that Theorem \ref{t:shearbasis} and Theorem \ref{t:rigidbracing} have counterparts for a finite parallelogram framework which is a maximal patch with simply connected boundary, where maximal means that there are no triple or quadruple of joints on the boundary which are the joints of a tile not in the patch.
Indeed in the inductive step from patch $\P$ to  patch $\P^+$ in Theorem \ref{t:shearbasis} there is a free choice of contiguous tile $T$ to add to $\P$. Moreover, for the same reason, the proof also applies to infinite patches which are simply connected and have the maximal property.

In their combinatorial approach to the finite case Graseggar and Legersk\'y \cite{gra-leg} consider a wider class of parallelogram frameworks, allowing crossing bars and nonplanar underlying structure. Their requirement on the underlying graph is that every ribbon is an edge cut in the sense that removing
the edges of the ribbon makes the graph disconnected.

Recently we have obtained characterisations of rigidity for finite and infinite braced grids %obtained a coloured braces graph  characterisation of braced grids 
in the plane with respect to some non-Euclidean norms \cite{pow-nonEuc}. It would be interesting to determine corresponding characterisations for braced parallelogram frameworks and general norms.
\end{rem}

\subsection{Linearly localised flexes}
A subset of vertices (resp. joints) of an embedded graph (resp. bar-joint framework) in $\bR^2$ is said to be \emph{linearly localised}, or $H$-localised, if every vertex (resp. joint) of the subset is located within a fixed distance of a line $H$ through the origin. Also a velocity field for a bar-joint framework in the plane is said to be $H$-localised if its support is $H$-localised.

\begin{thm}\label{t:bandltd_for_P}
Let $\G_P$  be a parallelogram bar-joint framework, with a choice of ribbon shears $u^\rho$ relative to a particular  base joint, and let $u$ be a velocity field. If
 $u$ is a nonzero $H$-localised infinitesimal  flex then it is a finite sum $u=\sum_{\rho  \in F} \alpha_\rho u^\rho$, over $H$-localised ribbons.
% , and
%$\sum_{\sigma\in F}\alpha_\sigma ${\bf t}$_\sigma = 0.$  
\end{thm}

\begin{proof} 
By Theorem \ref{t:shearbasis} an  infinitesimal flex  $u$ has an infinite sum representation \\ $u= \alpha_xu^x+\alpha_yu^y+\sum_\rho \alpha_\rho u^\rho$. 
Let $T$ be any tile with ribbons $\sigma, \tau$ and associated 4-bar framework $\T$. The restriction of any infinitesimal flex to $\T$ is a sum of an infinitesimal  translation and a unique linear combination of the restrictions of $u^\sigma$ and  $u^\tau$. Also the restriction of other ribbon shears to $\T$ are translation flexes. Thus, if the restriction of $u$ to $\T$ is zero, then $\alpha_\sigma = \alpha_\tau =0$.

Suppose now that $u$ is $H$-localised. If the ribbon $\rho$ is not $H$-localised then it travels outside the support of $u$ and from the previous paragraph $\alpha_\rho=0$.
It remains to show that the number of ribbons contained in a band of finite width is finite. This follows from our standing assumptions for $P$. These imply that there is a finite path of tiles which connects a pair of tiles on opposite sides of the ribbon.
\end{proof}

 For an $H$-localised ribbon $\rho$ in a general parallelogram tiling $P$ we may assign a unit vector {\bf n}$_{\rho}$ normal to  $H$, a unit vector {\bf m}$_{\rho}$ in the direction of the internal edges, and a unit vector {\bf t}$_{\rho}$ orthogonal to {\bf m}$_{\rho}$. These vectors are determined up to sign, and ${\bf t}_{\rho}$ or $-{\bf t}_{\rho}$ is the flex direction for all the nonzero velocities of a ribbon shear $u^\rho$ of $\G_P$. See Figure \ref{f:ribbonPartB}. 
We refer to  {\bf t}$_{\rho}$ and  $-{\bf t}_{\rho}$ as  \emph{tangential flex vectors} for the ribbon $\rho$.
%Additionally we define the deflection angle $0\leq \alpha(\rho) <\pi$

\begin{center}
\begin{figure}[ht]
\centering
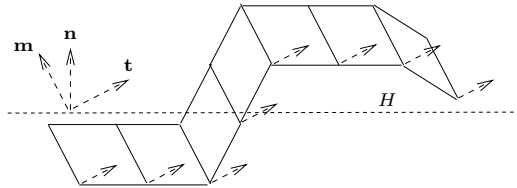
\caption{An $H$-localised ribbon $\rho$ with associated unit vectors {\bf n}$_{\rho}$, {\bf m}$_{\rho}$ and {\bf t}$_{\rho}$.
}
\label{f:ribbonPartB}
\end{figure}
\end{center}

%Suppose that $F$ is a finite family of $H$-localised ribbons and $u= \sum_{\sigma\in F}\alpha_\sigma u^\sigma$.
%WRONG Then it follows that $u$ is an $H$-localised infinitesimal flex if and only if $\sum_{\sigma\in F}\alpha_\sigma ${\bf t}$_\sigma = 0.$ In fact all $H$-localised infinitesimal flexes are of this form. 

%TEXT TO RECYCLE?? Let us say that two ribbons of $P$ are \emph{strictly separated} (resp. \emph{separated}) if they have no vertices (resp. tiles) in common. 
%or, equivalently, if they are disjoint as closed sets. A pair  $\rho, \sigma$  of separated, linearly localised ribbons defines an associated \emph{band of tiles} formed by the tiles of the ribbons $\rho, \sigma$ together with the tiles that lie between them. Write $B(\rho, \sigma)\subseteq \bR^2$ for the closed set formed by the union of these tiles. This is the \emph{closed band} for the pair of ribbons.

The following notation is useful for the next definition. For a line $H$ through the origin and $c>0$ let $U(H,c), V(H,c)$ be the closed connected sets whose union is the set of points whose distance to $H$ is at least $c$. 

\begin{definition}\label{d:bulkshear}
Let $\G$ be a countable bar-joint framework in $\bR^2$. 
A  \emph{bulk shear}, or \emph{$H$-localised bulk shear}, of  $\G$ is an infinitesimal flex $u$ with the following property. There are distinct velocity vectors $a, b$ in $\bR^2$ and a pair $H, c$ such that if $p_i\in U(H,c)$ (resp. $p_i\in V(H,c)$) then $u(p_i) = a$ (resp. $u(p_i) = b$). 
\end{definition}

\begin{thm}\label{t:bulkshear_for_P}
Let $\G_P$  be a parallelogram bar-joint framework with a choice of ribbon shears $u^\rho$ relative to a particular  base joint. Then a velocity field $u$ of $\G_P$ is an $H$-localised bulk shear if and only if it is a sum of a translational infinitesimal flex and a finite sum $\sum_{\rho  \in F} \alpha_\rho u^\rho$, over a set $F$ of $H$-localised ribbons.  
\end{thm}

\begin{proof}
Let $u$ be an $H$-localised bulk shear, with vector pair $a,b$ for the sets  $U(H,c), V(H,c)$.
By Theorem \ref{t:bandltd_for_P} $u$ has a representation $u_{\rm tr} + \sum_{\rho} \alpha_\rho u^\rho$, where $u_{\rm tr}$ is an infinitesimal translation. 
Let $T$ be a tile lying in $U(H,c)$ or $V(H,c)$. Then the restriction of $u$ to the associated bar 4-cycle $\T$ is a translation. As in the proof of Theorem \ref{t:bandltd_for_P}, it follows that $\alpha_\rho =0$ if the ribbon $\rho$ passes through $T$. Thus the summation for $u$ is the same as the summation over $H$-localised ribbons. Once again, by our standing assumptions for $P$, the sum is finite.
%OLD
%It suffices to show that if $u$ is an $H$-localised bulk shear with vectors $a, b$, and $b=0$, then $u$ is a finite sum of $H$-localised ribbon shears. By Theorem \ref{t:bandltd_for_P} $u$ is an infinite linear combination of ribbon shears from a free basis. It follows, as in the previous proof, that if a ribbon $\rho$ is not $H$-localised then its coefficient $\alpha_\rho$ in this expansion must be zero.
\end{proof}

\begin{center}
\begin{figure}[ht]
\centering
\includegraphics[width=4.2cm]{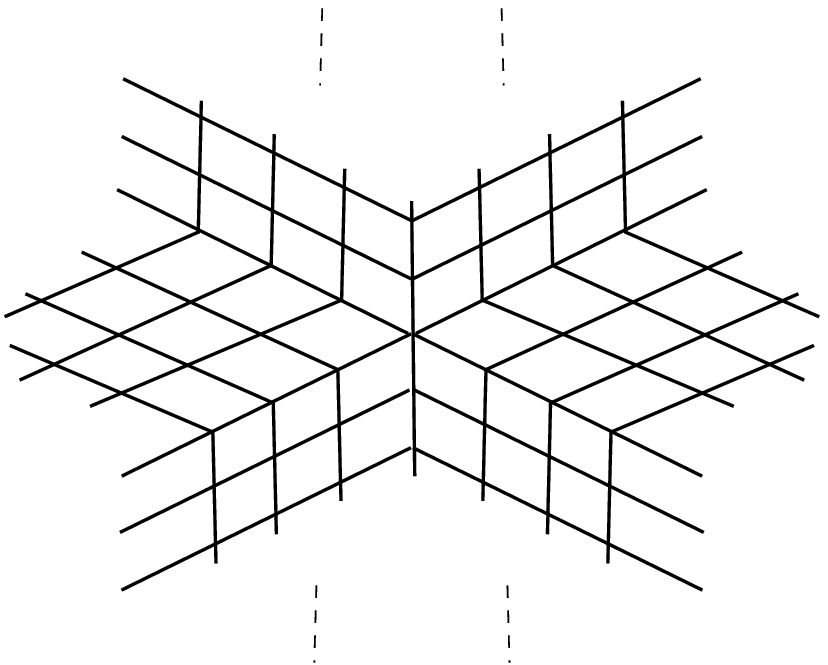}\quad \quad \quad \includegraphics[width=5cm]{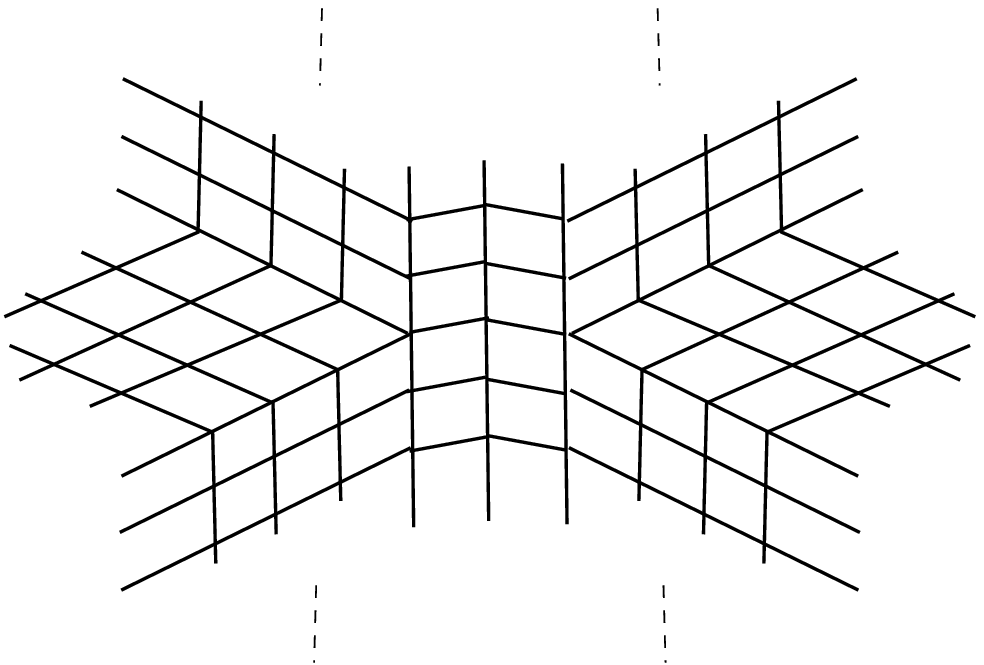}
\caption{Tilings with (a) no linearly localised ribbons, (b) two linearly localised ribbons.}
\label{f:bentribbons}
\end{figure}
\end{center}

We now define two elementary geometric invariants for a general countable bar-joint framework. These are particularly appropriate for parallelogram frameworks.

\begin{definition}\label{d:linefigures}
Let $\G$ be a countable bar-joint framework in $\bR^d$, for $d\geq 2$. 

(i) The \emph{linear flex figure}, LFF$(\G)$, is the union of the set of lines $H$ through the origin for which there exists a nonzero $H$-localised infinitesimal flex. 

(ii)  The \emph{bulk shear figure}, BSF$(\G)$, is the union of the set of lines $H$ through the origin for which there exists a nonzero $H$-localised bulk shear. 
\end{definition}

Note that if $\G$ has a locally supported infinitesimal flex then  LFF$(\G)= \bR^2$. 
On the other hand
Figure \ref{f:bentribbons} indicates parallelogram tilings with few linearly localised ribbons, the first having none and the second with two, say $\rho, \sigma$. In fact the second tiling gives a bar-joint framework $\G_P$ with no linearly localised infinitesimal flexes, since the tangential flex vectors {\bf t}$_\rho$ and {\bf t}$_\sigma$ are different. Nevertheless $\G_P$ has $H$-localised bulk shears and BSF$(\G_P)$ is equal to a vertical line through the origin.

 \section{Multigrid frameworks}\label{s:multigridG}
We now define parallelogram tilings $P$ that are associated in a dual way with sets of lines. In the case of those determined by multigrids it is shown that the ribbons are linearly localised and an explicit recipe is given for the \emph{ribbon figure} RF$(P)$, which is defined as the union of lines through the origin that record their directions. Also we identify various quasicrystal frameworks whose infinitesimal flex spaces are finite dimensional.

\subsection{Dual parallelogram tilings}
\label{ss:PfromPstar}

Consider a tiling $P_*$ of the plane by convex polygons that arises from a countable set of lines where no more than 2 lines intersect at a given point.  
We assume moreover that the angles of intersection of the lines, and the distances between parallel lines are bounded away from zero. 
We call such a tiling a \emph{regular line tiling} and it follows that the lines fall into finitely many equivalence classes of parallel lines. From $P_*$ we may construct parallelogram tilings $P$ by the following scheme, where $P_*$ represents a plan for the construction of tiles of $P$  from the vertices of $P_*$, and where the ribbons of $P$ correspond to the lines of $P_*$.

Start the construction with a choice of base face, $f_*$ say, of the regular line tiling $P_*$ and a choice of positive \emph{length} or \emph{weight}, $\lambda(H)$, for each line $H$ through the origin representing a class of parallel lines of $P_*$.
%We assume moreover that the weight function $\lambda:\{H\} \to \bR$  is bounded away from zero. 
If the face $f_*$ has $n$ edges then associate with it a base vertex $v$ (in a another copy of $\bR^2$) together with $n$ incident edges where these edges are orthogonal to the directions $H$ of the edges of $f_*$ and have lengths $\lambda(H)$. From this star graph, denoted $S(v)$, add pairs of edges to consecutive pairs of edges  to create $n$ parallelograms incident to $v$. In this way we create $n$ new vertices, $v_1, \dots ,v_n$. These vertices correspond to the faces of $P_*$ which are contiguous to $f_*$, say  $f_1^*, \dots ,f_n^*$. Thus, the vertex $v$ for the base face $f_*$ defines a base patch $P(v,1)$ which is determined in 2 steps, namely, the addition of edges to $v$, to create its star graph, followed by parallelogram completion. 

This 2-step construction can be repeated for the new vertices; 
%process can be repeated without obstacles to create a unique parallelogram tiling $P$, which we may also denote as $P(P_*,\lambda)$. 
the base patch is enlarged by completing the star graph $S(v_1)$ for  $v_1$ (some edges are already determined) together with its parallelogram completion. The result can be viewed as the join of the patch $P(v,1)$ with the patch $P(v_1,1)$  over their common edges. It follows from this that the process of star patch addition can 
be continued in a unique way  for the vertices $v_2, \dots , v_n$ (in any order). In view of our assumptions for $P_*$ there is a lower bound for the areas of constructed parallelograms and so the construction can be continued to create a unique parallelogram tiling $P$.   Specifically we may choose the order of star patch additions to follow an exhaustive enumeration of the faces of $P_*$ where consecutive faces of the sequence are adjacent. We denote the parallelogram tiling $P$, which is determined uniquely up to translation, as $P(P_*,\lambda)$.

\begin{definition}\label{d:dualparallelogramtiling}
A \emph{dual parallelogram tiling} is a parallelogram tiling of the form $P(P_*,\lambda)$ for some regular line tiling $P_*$.
\end{definition}

%In fact we can be more general in the construction and the  definition of a dual parallelogram tiling by choosing weight functions which do not have a constant value on parallel lines of $P_*$. This variation corresponds to widening ribbons of the same type, where the common length of the interior edges of the ribbons is varied.

Not all parallelogram tilings have preduals which are regular line tilings. This is the case for the tiling suggested by Figure \ref{f:ribbontwist}, where a crossing pair of ribbons lies between a noncrossing pair. On the other hand one can view general parallelogram tilings as constructed in a similar dual way relative to a network of curves running through the ribbons \cite{bee}.

\begin{center}
\begin{figure}[ht]
\centering
\includegraphics[width=5cm]{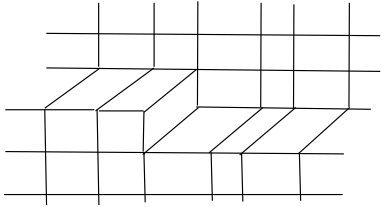}
\caption{Part of a parallelogram tiling $P$ that is not dual to a regular line tiling.}
\label{f:ribbontwist}
\end{figure}
\end{center}

\subsection{Multigrid quasicrystals}\label{ss:multigridQCs}
A significant family of dual parallelogram tilings arises from multigrids. The following definitions are convenient.
%BELOW:See also Gahler and Rhyner \cite{gah-rhy} where the equivalence with the projection method construction is given.

\begin{definition}\label{d:multigridtilings}
(i) A \emph{multigrid} in $\bR^2$ is a set of lines 
$
 A_1(\L)\cup \dots \cup A_r({\L}) 
$
where $\L$ is the set of lines parallel to the $y$-axis, with integer intercepts on the $x$-axis, and where $A_1, \dots , A_r$ are affine automorphisms of $\bR^2$. A multigrid is \emph{regular} if the component grids $A_i({\L})$ have distinct directions and every triple of lines has an empty intersection.
 
(ii) A \emph{multigrid parallelogram tiling} is a parallelogram tiling that is dual to a multigrid.
 
%(ii)  A \emph{multigrid parallelogram bar-joint framework} is a framework $\G_P$ where $P$ is a parallelogram tiling $P=P(P_*,\lambda)$ for some multigrid line tiling $P_*$ and weighting $\lambda$.
\end{definition}

%\begin{rem}
%{\color{magenta}to do: Comments} on Penrose rhomb tilings and others, projection method, equivalence, ... references... .
%\end{rem}

% It follows from the geometry of the arrowed tiles that there are $5$ types of ribbon, types 0 to 4 say, where type 0 has the common edges parallel to the $y$-axis and where types $k=1,2,3,4$ have their common edges parallel to the direction $\pi/2 + k\pi/5$. 
%[To state and  prove more rigourously] In general, if $P$ is a parallelogram tiling defined by a primitive substitution system \cite{baa-gri} then a ribbon associated with a common edge direction {\bf e} is linearly localised with hyperplane line $H$ whose direction is determined by the relative frequency of tiles in the subset of tile types for  {\bf e}. 
%For a Penrose tiling the hyperplane line of a ribbon is normal to {\bf e}. This follows since the 4 tile types of a ribbon fall into two pairs, the pairs being mirror images, relative to the common edge direction and are equally frequent. [Need to do more here]

%\begin{example}\label{e:penroseETC}
Define a \emph{regular de Bruijn pentagrid} to be the regular multigrid given by  $A_1,\dots , A_5$ where 
$A_k(z) = R_kz+(\gamma_k,0), z\in \bR^2,$ where $R_k$ is clockwise rotation by  $2\pi (k-1)/5$, and where $\gamma_1+\dots +\gamma_5=0$. The associated multigrid parallelogram tilings, with rhombic parallelograms, are examples of Penrose tilings. For these tilings $P$ the ribbon figure RF$(P)$ consists of 5 lines through the origin parallel to the 5 individual grid lines of the pentagrid.
\begin{center}
\begin{figure}[ht]
\centering
\includegraphics[width=7cm]{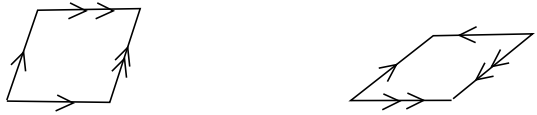}
\caption{Two arrow-decorated rhombs, with interior angles $72^\circ, 36^\circ$.}
\label{f:2tiles}
\end{figure}
\end{center}

To be more precise about the nature of Penrose tilings, consider the thick and thin rhomb tiles of Figure \ref{f:2tiles}, with base edges parallel to the $x$-axis, each with an arrow decoration. Two further decorations arise by rotation by $\pi$ and a totality of 20 decorated rhomb tiles is obtained by including the 16 further decorated tiles obtained by rotations by  $2(k-1)\pi/5$, for $2\leq k\leq 5$. 

\begin{definition}\label{d:penrosetiling}
A \emph{decorated Penrose rhomb tiling}, or arrowed rhomb tiling (AR tiling), is a decoration-matched tiling of the plane by thick and thin decorated rhombs. A \emph{Penrose rhomb tiling} is a tiling obtained by undecorating an AR tiling.
\end{definition}

In 1981 de Bruijn \cite{deB} defined Penrose rhomb tilings in this way and showed how they are related to  (possibly nonregular) pentagrid parallelogram tilings. 
Also he obtained fundamental properties, such as aperiodicity, local isomorphism and substitution rules, and showed the connection with the projection method construction via lattices in $\bR^5$. Penrose's original inflation construction gives a particular subset of the AR tilings. An overview of de Bruijn's work and connections with physical quasicrystals and diffraction is given by Au-Yang and Perk \cite{au-per}.

Beenker \cite{bee} extended de Bruijn's algebraic methods to tetragrids, that is, to 4-fold grids. These multigrids give rise in particular to the so-called Ammann-Beenker tilings, or octagonal tilings, with 2 tile types up to congruency, namely a square and a $\pi/4$ rhombus. Despite the simplicity of their construction there are 41 different types of vertex up to translation. 
Socolar \cite{soc} subsequently gave a detailed geometric analysis including the decagonal case, with tilings by hexagons, squares and $\pi/12$ rhombs.  
Moreover, Gahler and Rhyner \cite{gah-rhy} have given the equivalence between a somewhat more general construction of parallelogram tilings for multigrids and the projection method construction. 
%\end{example}
\medskip

The following observations follow from  Theorem \ref{t:rigidbracing}.

\begin{cor}
(i) Let $P$ be a regular Penrose rhomb tiling and let $\B$ be the set of thick tiles or the set of thin tiles. Then the braced framework $\G_P(\B)$, obtained from $\G_P$ by bracing the tiles of $\B$, is infinitesimally rigid. (ii) Let $P$ be an Ammann-Beenker tiling and let $B$ be the set of square tiles. Then the braced framework $\G(\B)$ is infinitesimally flexible.
\end{cor}

\begin{proof}
(i) Suppose $\B$ is the set of thin tiles. The ribbons of $P$ correspond to the grid lines of the regular multigrid $P_*$, and a pair of ribbons meet at a thin tile if and only if their grid lines meet at angle $\pi/5$. These ribbons represent vertices in the braces graph $G(\B)$  that are connected by an edge. Since any 2 grid lines either meet at angle $\pi/5$ or have an intermediate grid line at angle $\pi/5$ to each of them, it follows that $G(\B)$ is connected and spanning. A similar argument holds when $\B$ is the set of thick tiles.

(ii) The square tiles correspond to the edges of the braces graph. However, there are 2 types of square up to translation and each ribbon only contains squares of one type. It follows that $G(\B)$ is not connected and so, by Theorem \ref{t:rigidbracing}, $\G(\B)$ is infinitesimally flexible.
\end{proof}

In the next lemma we show that ribbons are necessarily  linearly localised. The proof gives a formula for the directions of the ribbons and this is summarised in Proposition \ref{p:ribbondirection}.

\begin{lemma}\label{l:multigridPribbons}
Every ribbon of a regular multigrid parallelogram tiling is linearly localised.
\end{lemma}

\begin{proof}
Let $\rho$ be a ribbon, of a regular multigrid parallelogram tiling  $P$, which corresponds to a line $L$ of the component grid $A_1(\L)$ of the predual multigrid  $P_*$. Assume that $L$ is not a vertical line and and is parametrised by arclength by the real variable $s$ measured from a  vertex of $P_*$, 
and let $T_1$ be the tile of $\rho$ corresponding to this intersection point. Then the set $X_{1,j}$ of parameter values $s$ for points of intersection of $L$ with lines of a grid $A_j(\L)$, with $j\neq 1$, has the form
\[
X_{1,j}=\{\beta_{1,j}n+\gamma_{1,j}:n\in \bZ\},
\]
where, without loss of generality, $\beta_{1,j}>0$ and $ \gamma_{1,j}\geq 0$. 

Note that the points of the parameter set $X_{1,j}$ correspond to the appearance of congruent tiles in the ribbon $\rho$ and we say that these tiles are of type $(1,j)$. Here $2 \leq j \leq r$ where $r$ is the number of grids for $P_*$. Viewing these sets as containing points of the same colour,
the multi-coloured set $X_1$, formed by the union of the distinctly coloured sets $X_{1,j}$, encodes the sequential appearance of the tile types of $\rho$. 

To show that $\rho$ is linearly localised, introduce a copy $L'$ of $L$ in the separate ambient space for $P$, with its origin, for parameter value $s=0$, located at a vertex $v_1$ of $T_1$. 
%(Formally, the lines $L$ and $L'$ are in different ambient spaces.)
Let $T_1, T_2, \dots $ be the tiles of $\rho$ in sequential order corresponding to the positive $s$-direction of $L'$ and, rechoosing $v_1$ if necessary, arrange that $v_1$ is not a vertex of $T_2$.  Also let $M'$ be a line through this vertex which is  orthogonal to $L'$, with arclength parametrised by $t$. The vertex $v_1$ has $(s,t)$-coordinates $(0,0)$ and lies on a two-way infinite piecewise linear boundary curve $C$ of the ribbon $\rho$. We show that $C$ is linearly localised and this will complete the proof.

Let $v_1, v_2, \dots $ be the vertices on $C$ in their sequential order with $(s,t)$-coordinates $(s_k,t_k), k=1,2, \dots.$
Let $T$ be a tile of $\rho$ of type $(1,j)$. It has an edge $e$ which is not parallel to $M'$ and we define
the \emph{$s$-increment} (resp. $t$-increment) of $T$, for $\rho$, to be the length $s_{1,j}$ (resp. $t_{1,j}$) of the projection of $e$ onto $L'$ (resp. $M'$). See Figure \ref{f:ribbonincrements}.
\begin{center}
\begin{figure}[ht]
\centering
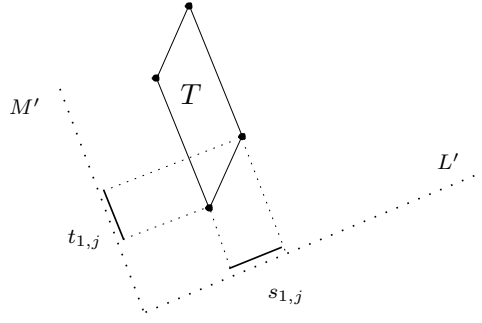
\caption{Increments $s_{1,j}>0$ and $t_{1,j}\in \bR$ associated with a tile $T$ of type $(1,j)$ in a ribbon associated with a grid line of the grid $A_1(\L)$.}
\label{f:ribbonincrements}
\end{figure}
\end{center}

Let $s_*$ be a parameter value for the line $L$. Then, for $s_* >0$ the finite coloured set $X_1\cap [0,s_*]$ encodes a finite stretch of the ribbon $\rho$. For each $j,$ we define  the subset cardinalities 
\[
n_{1,j}(s_*)=|X_{1,j}\cap [0,s_*]| =|\bZ\cap [\frac{-\gamma_{1,j}}{\beta_{1,j}}, \frac{s_*-\gamma_{1,j}}{\beta_{1,j}}]|.
\]
Then the coordinates $(s_{k+1},t_{k+1})$ of the final vertex $v_{k+1}$ for this stretch of $C$ are
\[
s_{k+1}= \sum_{j=2}^r n_{1,j}(s_*)s_{1,j}, \quad t_{k+1}= \sum_{j=2}^r n_{1,j}(s_*)t_{1,j}. 
\]
For each $j$ we have
\[
n_{1,j}(s_*) = \frac{s_*}{\beta_{1,j}} +o(s_*), \quad \mbox{as}\quad s_*\to \infty,
\]
and so
\[
\frac{t_k}{s_k} \to m, \quad \mbox{as}\quad k\to \infty,\quad \mbox{where}\quad m= \frac{\sum_{j=2}^r\delta_{1,j}t_{1,j}}
{\sum_{j=2}^r\delta_{1,j}{s_{1,j}}},
\]
and where $\delta_{1,j}$ is the relative frequency of the colour $j$, namely the ratio
\[
\delta_{1,j} =\frac{\beta_{1,j}^{-1}}{\sum_{j=2}^r\beta_{1,j}^{-1}}.
\]

The same argument holds for $s_*<0$, and it follows that the ribbon $\rho$ is $H$-localised for the line $H$ through the origin with gradient $m$.
%OLD
%Also, for each $j,$ 
%\[
%|X_{1,j}\cap [0,s_*]| =\lfloor\frac{1}{\beta_{1,j}}(s_*-\gamma_{1,j})\rfloor 
%\]
%and so we can compute the coordinates $(s_k,t_k)$ of the final vertex $v_k$ for this stretch;
%\[
%s_k= \sum_{j\in J_1} \lfloor\frac{1}{\beta_{1,j}}(s_*-\gamma_{1,j})\rfloor s_{1,j}, \quad \quad t_k= \sum_{j\in J_1} \lfloor\frac{1}{\beta_{1,j}}(s_*-\gamma_{1,j})\rfloor t_{1,j}.
%\]
%It follows that the ribbon $\rho$ is $H$-localised for %the line $H$ given by $t=ms$ where 
%\[
%m=\frac{\sum_{j\in J_1}\frac{t_{1,j}}{\beta_{1,j}}}
%{\sum_{j\in J_1}\frac{s_{1,j}}{\beta_{1,j}}}
%\]
%TO COMPLETE USING $\lfloor x \rfloor $ etc.
\end{proof}

\begin{definition}\label{d:ribbonfigure}
The \emph{ribbon figure}, or \emph{ribbon line figure}, RF$(P)$ of a parallelogram tiling $P$ is the set of lines $H$ through the origin for which there exists an $H$-localised ribbon. 
\end{definition}

\begin{thm}\label{t:FiguresformultigridGP}
Let $\G_P$ be the bar-joint framework of a regular multigrid parallelogram tiling $P$. 
%whose multigrid has $r$ component grids. 
Then the bulk shear figure BSF$(\G_P)$ and the linear flex figure LFF$(\G_P)$ are equal to the ribbon figure RF$(P)$. 
\end{thm}

\begin{proof}
This follows from Theorem \ref{t:bulkshear_for_P}.
\end{proof}

%It follows from the proof of Theorem \ref{t:FiguresformultigridGP} that

For any set $\M$ of lines, not necessarily passing through the origin, let LF$(\M)$, the \emph{line figure} of $\M$, be the set of lines through the origin which are parallel to some line of $\M$.  
For a parallelogram or rhombus tiling $P$, given by a multigrid $P_*$, the line figure LF$(P_*)$ need not agree with the ribbon figure RF$(P)$. The details for this, in terms of increments of projections, are given in the proof of Lemma \ref{l:multigridPribbons} and summarised in the following proposition.

\begin{prop}\label{p:ribbondirection}
Let $P=P(P_*,\lambda)$ where $P_*$ is a regular multigrid, with $r$ nonparallel component grids $P_*^i$ and line figure LF$(P_*)=\{L_1,\dots , L_r\}$.
%,  with $\lambda(H_i)$ equal to the edge length of edges of $P$ orthogonal to $H_i$. 
For $1\leq i \leq r$, and for each tile $T_{i,j}, j\neq i,$ associated with the pair $L_i, L_j$, let $s_{ij}>0$ (resp. $t_{ij}\in \bR$) be the length of the projection
onto $L_i$ (resp. $L_i^\perp$) of the edge of the tile which is not orthogonal to $L_i$. Also, let $\delta_{i,j}$, for $j\neq i$, be the relative frequencies in $L_i$ of the points of intersection with $P_*^j$, with $\sum_{j:j\neq i}\delta_{i,j} =1$.  Then the ribbon for $L_i$ is $H$-localised, where the tangent of the angle between $H$ and $L_i$ is
\[
\frac{\sum_{j:j\neq i}\delta_{i,j}t_{i,j}}
{\sum_{j:j\neq i}\delta_{i,j}{s_{i,j}}}
\]
\end{prop}

The multigrid line figure and the ribbon figure may coincide for reasons of symmetry and this is so 
for the Penrose rhomb tilings. Specifically, for a given ribbon direction associated with a grid line $L$ the 4 other intersecting grid line types can be paired (by bilateral symmetry with respect to an orthogonal line to  $L$), and their associated rhombs at intersection points similarly paired. These pairs of rhombs, with bilateral symmetry, have cancelling $t$-increments.

%The zero modes (rigid unit modes) and zero mode spectrum of a crystallographic framework $\C$ are determined by a finite geometric data set, for a repeating block of bars and joints, and by phase variations over translated blocks. Despite this limitation they nevertheless capture many aspects of the flexibility of $\C$.  
%This would amount to a quasicrystallographic analogue of the role of the Indeed, as shown in  Badri, Kitson and Power \cite{bad-kit-pow-1}, the RUM spectrum can be viewed as a generalised Bohr spectrum for which the associated zero modes have dense linear span in the space of almost periodic first-order motions. It is of interest then to determine analogous spectra, or partial analogues, in the case of quasicrystallographic bar-joint frameworks. 

\subsection{Finitely flexible quasicrystals}
 %These are braced parallelogram frameworks associated with the associated with what we shall refer to as \emph{checkerboard quasicrystals}.
Let $P$ be the dual parallelogram tiling $P(P_*,\lambda)$ defined by the regular multigrid $P_*$
with component grids $A_i(\L)$, for $1\leq i\leq r$, as in Definition \ref{d:multigridtilings}. For each pair of distinct indices $i, j$ let $W(i,j)$ be the lattice of vertices of $P_*$, when viewed as an embedded graph, which is determined by the component grids for $i,j$. Let $W(i,j)= \{w(i,j,k,l):(k,l)\in \bZ^2\}$
where the labelling corresponds to the lines of the grid $A_i(\L)$ (resp. $A_j(\L)$) being ordered by indices $k\in \bZ$ (resp. $l\in \bZ$). For definiteness we assume that the linear part of each $A_i$ has positive determinant and the labelling is inherited from the $x$-coordinate labelling of the lines of $\L$.
%In particular, these orderings determine $r(r-1)/2$ origin vertices $w(i,j,0,0)$. 
The parallelogram tiles of $P$ have a corresponding labelling, $T(i,j,k,l)$, for $ i\neq j, (k,l)\in \bZ^2$.

\begin{definition}\label{d:semibracedquasicrystal}
A \emph{checkered multigrid parallelogram tiling}, with integer \emph{density index} $p\geq 2$, is a pair $(P,\B)$ where $P$ is a regular multigrid parallelogram tiling and $\B$ is the set of tiles 
\[
\B= \{T(i,j,k,l):1\leq i,  j\leq r, i\neq j, (k,l)\in \bZ^2, k=l\mod p\}.
\]
A \emph{checkered quasicrystal framework} is a bar-joint
framework $\G_P(\B)$ for such a pair which is obtained from the parallelogram framework $\G_P$ by adding bracing bars for the tiles of $\B$.
\end{definition}

If the tiles of $\B$ are considered as being coloured black, rather than the default colour white, then for  $p=2$, the coloured tiling $(P,\B)$ 
%the bar-joint framework $\G_P$ can be viewed as a 
is a quasicrystal counterpart of an infinite checkerboard. Accordingly in this case we refer to $\G_P(\B)$ as a \emph{checkerboard quasicrystal}.

%For a regular de Bruijn pentagrid tiling $P$, the brace bars of $\G_P(\B)$ can be added in a consistent manner, with long diagonals for thick rhombs and short diagonals for thin rhombs. This checkerboard quasicrystal then corresponds to an aperiodic tiling by Penrose rhombs and Robinson triangles. Alternatively one can simply view the black tiles as rigid tiles. 

%LOOSE TEXT PUT BELOW IN REMARKFor index $p=2$ this can be viewed as an analogue of the braced....  Indeed, both frameworks have infinitesimal flex spaces of dimension 4, although the counterpart to the alternating rotation flex of braced squares is not so evident. 
%{\color{red}add about predual reason}

It is straightforward to see that there exist $p-1$ tiles which when added to the black set $\B$ create a bracing set $\B^+$ whose braces graph $G(\B^+)$ is connected and spanning. By Theorem \ref{t:rigidbracing}, $\G(\B^+)$ is infinitesimally rigid and $\dim \F(\G(\B))\leq 3+(p-1).$ In fact we have equality.

\begin{thm}\label{t:fintelyflexible}
Let $\G_P(\B)$ be a checkered quasicrystal framework determined by a pair $(P,\B)$ with index $p\geq 2$.
Then the infinitesimal flex space $\F(\G_P(\B))$ has dimension $p+2$.
\end{thm}

\begin{proof} The braces graph $G(\B)$ is spanning since on every ribbon of type $(i,k)$, determined by the $k^{th}$ line of the grid $A_i(\L)$ of $P_*$, there is a tile 
$T(i,j,k,l)$ in $\B$ for some $j\neq i$ and $l\in \bZ$. The braces graph has $p$ components, say
$G_t, 0\leq t \leq p-1,$ where the vertices/ribbons of $G_t$ are of type $(i,k)$ with $k =t \mod p$. 
Let $p_1$ be a fixed joint and let
$\F(\G_P(\B))_0$ be the subspace of infinitesimal flexes $z\in \F(\G_P(\B))$ with $z(p_1)=0$. Note that every infinitesimal flex of $\G_P(\B)$ is an infinitesimal flex of the parallelogram framework $\G_P$. By Theorem \ref{t:shearbasis} it follows that each $z\in \F(\G_P(\B))_0$ has a unique representation
\[
z=
%\sum_{t=0}^{p-1}z_t= 
\sum_{t=0}^{p-1}\sum_{n=1}^\infty \alpha_{t,n} u_{t,n}
\]
where $\{u_{t,n}:n\in \bN, 0\leq t \leq p-1\}$ is a set of ribbon shears, associated with $p_1$ and the set of ribbons $\{\rho_{t,n}:n\in \bN, 0\leq t \leq p-1\}$ of $\G_P$. Here the index $t$ indicates the component $G_t$ of $G(\B)$ for which the $t$-labelled ribbons are vertices.

Since the subgraphs  $G_t$ are the connected components of $G(\B)$ it follows, as in the proof of Theorem \ref{t:rigidbracing}, that the velocity fields 
\[
z_t = \sum_{n=1}^\infty \alpha_{t,n} u_{t,n}
\] 
are also infinitesimal flexes of $\G_P(\B)$. Moreover, since $G_t$ is connected the coefficients $\alpha_{t,n}, n=2,3,\dots ,$ are determined by $\alpha_{t,1}$ for each $t$, and so the subspace of infinitesimal flexes $z_t$ is either 1-dimensional or the zero subspace. It remains to show that they are 1-dimensional. This follows since if $z$ is an infinitesimal rotation flex about $p_1$ then its component flexes $z_t$ are nonzero. 
\end{proof}
%\end{example}

\begin{rem}\label{r:singlezeromode}
The alternatingly braced squares framework, for the usual infinite checkerboard, is perhaps the paradigm example of a crystallographic framework with a single nontrivial zero mode (up to scalar multiples). See, for example, Figure 2 and Example (f) of Power \cite{pow-poly}, and Figure 6 of Dove \cite{dov-2019}. The framework  also features as cross-sections of the three dimensional perovskite framework of corner connected octahedra, and is responsible for lines in the zero mode spectrum.
\end{rem}

\begin{rem} As we have noted in the introduction, bar-joint frameworks provide a fundamental model for network materials and it seems to us that explicit free basis methods can provide a useful new perspective. In particular this viewpoints avoids considerations of periodic approximants, or periodic boundary conditions, which can be problematic \cite{mou-nau}.
\end{rem}

%\subsection{Zero modes}

\section{Zero mode spectra for quasicrystals}\label{ss:quasicrystalbarjoint}

%Let us also extend the definition of a \emph{checkered quasicrystal framework) to cover braced parallelogram frameworks $\G_P(\B)$ where $\B$ is a set of tiles determined by a nonempty proper subset of pairs $i,j$ with $1\leq i, j\let r$.

%Returning to unbraced parallelogram frameworks we have the following identification of their linear flex figures.

%For crystallographic frameworks in $\bR^2$ it is well-known that linearly localised flexes lead to linesin the zero mode (RUM) spectrum \cite{bad-kit-pow-2}, \cite{dov-2019}, \cite{pow-poly}. 

%SORT $\bZ^5$ labelling etc pre PHASE FIELD  

The notion of a zero mode, as opposed to a floppy mode (infinitesimal flex), of an infinite bar-joint framework in $\bR^d$ is that of an excitation state (first-order simple harmonic motion oscillatory state) with an associated wave vector {\bf k}. For a crystallographic framework $\C$ these wave vectors live in the reciprocal space, $\bR^d_{\bf k}$ say, determined by a basis $\ul{a} =\{a_1, \dots ,a_d\}$ of periodicity vectors for $\C$. The periodic reduction of the set of wave vectors is the RUM spectrum (or reduced zero mode spectrum) for the pair $\C, \ul{a}$, denoted $\Omega(\C,\ul{a})$, and this can be defined more directly, as we do below, in terms of complex infinitesimal flexes that are periodic up to a multiphase factor \cite{owe-pow-crystal}, \cite{pow-poly}. In this section we give a definition of a zero mode spectrum for some quasicrystal frameworks, and which are motivated by results in the previous sections.

%braced multigrid parallelogram frameworks and, in particular, for checkered quasicrystals.

\subsection{Zero mode spectra for crystals} We first summarise the formulation of a zero mode of a crystal framework in terms of \emph{phase fields} for a cell partition of the ambient space. 
%[SKIP/MOVE "The generalisation for checkerboard quasicrystals has a similar formulation in terms of generalised phase fields".]
% for a generalised partition determined by the underlying multigrid.

%In the absence of a periodicity basis we can in fact choose an arbitrary, or convenient, reference basis $\ul{a}$ and seek wave vectors in $\bR^d_{\bf k}$ (once again the dependence on the basis is understood) with respect to \emph{generalised phase fields}. What we mean by this, in part, is that in the crystallographic case

A \emph{cell partition} for a basis $\ul{a}$ of $\bR^d$ is a partition $\P = \{C_k:k\in \bZ^d\}$ where $C_k$ is the semi-open parallelepiped
\[
C_k = [k_1a_1,(k_1+1)a_1)\times \dots \times [k_da_d,(k_d+1)a_d).
\] 
A wave vector ${\bf k} = (\gamma_1,\dots , \gamma_d)$ corresponds to the {phase field} map $\phi_{\bf k}: \bR^d \to \bC$ given by
\[
\phi_{\bf k}(x) = 
%e^{2\pi i({ \gamma_1}x_1+\dots +{ \gamma_d}x_d)} =
 e^{2\pi i{\gamma_1}k_1}\dots e^{2\pi i{ \gamma_d}k_d}, \quad x \in C_k, \quad k\in \bZ^d.
\]
%where the sets $C(t_1,\dots ,t_d)$, for $ t=(t_1,\dots ,t_d)\in \bZ^d$, is the cell partition of the ambient space determined by $\ul{a}$.
This phase field, in physical space, is constant on each individual cell of the partition. Let $\C$ be a crystallographic framework with a set of periodicity vectors $\ul{a}=\{a_1,\dots , a_d\}$ which is a basis for $\bR^d$. Define an associated phase-periodic velocity field
$u$, for the wave vector ${\bf k}$, to be a map from the set of joints of $\C$ to velocity vectors in $\bC^d$, with the property
\[
u(T_k(p_i))=\phi_{\bf k}(T_k(p_i))u(p_i), \quad p_i \in C_0,
\]
where $T_k(p_i)$ is the joint in $C_k$ which is the translate of the joint $p_i$, in the cell $C_0$, given by
\[
T_k(p_i)= p_i+ k_1a_1+\dots +k_da_d.
\]
The $d$-tuple
\[
\omega = (\omega_1, \dots , \omega_d) = 
 (e^{2\pi i{\gamma_1}},\dots , e^{2\pi i{ \gamma_d}})
\]
is the \emph{multiphase} of $u$ and, with the usual multinomial notation convention, we have
\[
u(T_k(p_i))=\omega^k u(p_i), \quad k \in \bZ^d, p_i\in C_0.
\]

\begin{definition}\label{d:rumspec}\label{e:gridRumSpec}
The rigid unit mode spectrum, or RUM spectrum, $\Omega(\C,\ul{a})$, of a crystallographic bar-joint framework $\C$ and  periodicity basis $\ul{a}$, is the set of multiphases $\omega$ for which there exists a nonzero phase-periodic complex infinitesimal flex. The unreduced zero mode spectrum is the set of wave vectors ${\bf k}$ whose  multiphase $\omega$ is a point of $\Omega(\C,\ul{a})$.
\end{definition}

\begin{example}\label{e:gridexample} For a simple but relevant illustration let $d=2$ and let $\C_{\bZ^2}$ be the parallelogram bar-joint framework for the tiling by squares where the joints have integer coordinates. The basis $\ul{a}=\{(1,0),(0,1)\}$ is a periodicity basis and each set $C_k, k\in \bZ^2$, contains a single joint, say $p_k$. For each $\lambda \in \bT$ the velocity fields
\[
u_{x,\lambda}(p_k) = \lambda^{k_2}(1,0), \quad 
u_{y,\lambda}(p_k) = \lambda^{k_1}(0,1), \quad k=(k_1,k_2)\in \bZ^2,
\]
are infinitesimal flexes and so $\Omega(\C,\ul{a})$ contains the set $ (\{1\}\times \bT) \cup (\bT\times\{1\})$.
Note moreover, that the complex infinitesimal flex vector space of $\C_{\bZ^2}$ has a free basis of complex infinitesimal flexes which are supported on horizontal and vertical lines of joints, and $u_{x,\lambda}$ (resp. $u_{y,\lambda}$)
is an infinite linear combination of the horizontal (resp. vertical) line flexes. 
One can use this free basis to show that there are no other points in the RUM spectrum. To see this, consider the free basis of infinitesimal flexes $u_{x,t}, u_{y,t}$, for $t\in \bZ$, where
$u_{x,t}(p_k)= (1,0)$ if $t=k_2$, and $u_{y,t}(p_k)= (0,1)$ if $t=k_1$. Suppose that $z$ is a zero mode for $\lambda= (\lambda_1, \lambda_2)$. Then $z$ has a representation
\[
z = \sum_t \alpha_{x,t}u_{x,t} + \alpha_{y,t}u_{y,t}.
\]
Thus
\[
\lambda_1^i\lambda_2^jz(p_{0,0}) = z(p_{i,j}) = 
 \alpha_{x,j}(1,0) + \alpha_{y,i}(0,1) =
 (\alpha_{x,j},\alpha_{y,i}).
\]
Since $z$ is nonzero the velocity vector $z(p_{0,0})$ is also nonzero. Now the equations above, for $0\leq i,j\leq 1$, imply that at least one of $\lambda_1, \lambda_2$ is equal to 1.

For a direct general approach to the determination of  $\Omega(\C,\ul{a})$, 
one considers the infinitesimal flex equations for a representative set of bars, such as the bars in the building block motif. For the pair $\C,\ul{a}$ these equations are satisfied for a nonzero phase-periodic velocity field $u$ of the form
%with phase factor $(\lambda_1, \lambda_2)$ 
\[
u(p_k) = \lambda_1^{k_1}\lambda_2^{k_2}(a,b),
\]
if the conjugate $(\ol{\lambda_1},\ol{\lambda_2})$  is a point of rank degeneracy of a $2 \times 2$ matrix-valued function $\Phi(z_1,z_2)$ on $\bT^2$, with $(a,b)$ in the null-space of $\Phi(\ol{\lambda_1},\ol{\lambda_2})$. This also leads to the identification of $\Omega(\C,\ul{a})$ as $( \{1\}\times \bT) \cup (\bT\times\{1\})$. 

It is also convenient, and common in applications, to record the RUM spectrum as a reduced set of wave vectors. For this example the spectrum is represented as the subset of the unit square $[0,1)\times [0,1)$ of points $(\gamma_1, \gamma_2)$ with $\gamma_1=0$ or $\gamma_2=0$. 

The matrix function is known as the \emph{symbol function} of the crystal framework with respect to the periodicity basis $\ul{a}$. See  \cite{owe-pow-crystal}, \cite{pow-poly} and \cite{kas-kit-mcc} for further details.
% (for $\C_{\bZ^2}, \ul{a})$).
%ELSEWHERE? For a general periodicity basis $\ul{a'}=\{a_1', a_2'\}$ there may be $n$ joints in the unit cell $C_{(0,0)}$ in which case the associated  symbol function has $2n$ rows determined by a choice of $2n$ representative bars for the $2n$ equivalence classes for translations $T_k, k\in \bZ^2$.  
\end{example}

\subsection{Zero mode spectra for quasicrystals}
In the absence of an underlying periodic structure for an aperiodic framework we consider some notions of zero mode spectra which relate to generalised phase fields. 
%which are not necessarily commensurate with a single basis in the ambient space. 
In the case of multigrid parallelogram frameworks $\G_P$ we define zero modes in terms of multiparameter phase fields.
As in the crystallographic case these modes are in fact determined by the points $\omega$ of rank degeneracy of a multi-variable matrix-valued function and the associated null space vectors. However, the independent variables range in an $r$-torus, $\bT^r\subset \bC^r$, rather than a 2-torus, where $r$ is the number of component grids of the associated multigrid $P_*$. 

Consider a regular multigrid $P_*$, in the ambient space $\bR^2_*$, with $r$ component grids, no two of which are parallel. Each  grid  defines a partition of $\bR^2_*$ by semi-open bands. If $\A=\{A_1, \dots , A_r\}$ is the set of affine transformations determining this multigrid then a partition of $\bR^2$ is given by the images under $A_i$ of the partition by vertical left-closed bands $B_m, m\in \bZ$, associated with the integer lattice $\L$, where $B_m= [m,m+1)\times \bR$. Taking the intersection of these $r$ partitions gives a partition $C(P_*)$
%= \{C_k, k\in \bZ^r\}$ 
of $\bR^2_*$ by semi-open polygonal regions. Each set of the partition has the form $C_m= A_1(B_{m_1}) \cap \dots  \cap A_r(B_{m_r})$ for some $r$-tuple $m=(m_1,\dots , m_r)$. Let $\M$ be the subset of $\bZ^r$ consisting of these $r$-tuples. The interiors of the polygon sets give a partition of the complement of $P_*$ in $\bR^2_*$, and their closures are the faces (or tiles) of $P_*$, when $P_*$ is viewed as a regular line tiling. 
%Let $F(P_*)$ be the set of these closed faces. 
Since faces, $f$ say, of $P_*$ correspond to vertices of $P$ we may denote the set of joints of $\G_P$ as 
$\{p_f: f \in F(P_*)\}$, or $\{p_m: m \in \M\}$.

% and note that this is a labelling by the semi-open sets of the partition.

\begin{definition}\label{d:qc_phasefield} Let $\G_P$ be a parallelogram framework for the regular multigrid $P_*$.
For each $r$-tuple  $\omega = (\omega_1,\dots ,\omega_r)$ in $\bT^r$ the phase field $\phi_\omega$ on the set of joints  $\{p_m:  m \in \M\}$ is given by 
$
\phi_\omega(p_m) = \omega^m= \omega_1^{m_1}\dots \omega_r^{m_r}, $ for 
$m\in \M$.
%\quad \mbox{for} \quad C_\alpha= A_1(B_{k_1})\cap \dots  %\cap A_r(B_{k_r}).  
%\]
\end{definition}

A regular multigrid framework $\G_P$ has finite local complexity and indeed the joints are of finitely many translation types according to the translation types of the star graphs of vertices of $P$. Suppose that there are $n$ translation types and let $p_{f_1}, \dots , p_{f_n}$ be a choice of representatives.
%, indexed by translation representatives $\alpha_1, \dots , \alpha_n$.
Also write $p_f \equiv p_{f_i}$ if $p_{f}$ is of the same translation type as $p_{f_i}$.

\begin{definition}\label{d:qc_velocityfield} Let $\G_P$ be a parallelogram framework of a regular multigrid with $r$ component grids given by the set of affine transformations $\A=\{A_1, \dots , A_r\}$, and let $p_{f_1}, \dots , p_{f_n}$ be a set joints representing the $n$ possible translation types of joints.

(i) A \emph{phase-periodic velocity field} $u$ for $\G_P$, with phase factor $\omega\in \bT^r$, is a complex velocity field such that
\[
u(p_f) = \phi_\omega(p_f)b_i, \quad \mbox{if} \quad p_f \equiv p_{f_i},
\] 
where $b_1, \dots , b_n$ is a set of velocity vectors in $\bC^2$. 

(ii) A \emph{zero mode} of $\G_P$, with phase factor $\omega$, is a phase-periodic velocity field for $\omega$  that is an infinitesimal flex.

(iii) The \emph{reduced zero mode spectrum} of $\G_P$ is the set $\Omega(\G_P,\A)$ of phase factors $\omega \in \bT^r$ for the zero modes of $\G_P$.
\end{definition}

In this definition the phase fields
$\phi_\omega$ are defined only on the set of joints. However, these fields may be viewed as restrictions of phase fields on $\bR^2$ that are constant on the sets of a partition of $\bR^2$ induced by the partition $C(P_*)$, as we now indicate. In particular the spectrum $\Omega(\G,\A)$ may be defined in the same way if $\G$ is obtained from $\G_P$ in a systematic way by local moves with bars and joints. 
%This applies in particular to the checkered quasicrystal frameworks.

%There is a natural way to lift this partition and phase field for the ambient space $\bR^2_*$ of $P_*$ to a partition and phase field in the ambient space $\bR^2$ for $P = P(P_*,\lambda)$. 

The regular multigrid $P_*$ defines an $\bR^2_*$-embedded graph $G(P_*)=(V(P_*),E(P_*)$) where $V(P_*)$ is the set of intersection points of the grids of $P_*$, and $E(P_*)$ is the set of line segments joining consecutive vertices on the lines of the grids.
Also, in $\bR^2$ the \emph{skeleton} sk$(P)$ of $P$ is defined as the union of the piecewise linear curves formed by the line segments joining the midpoints of opposite edges of the tiles of $P$. Using these curves define the $\bR^2$-embedded graph (the skeleton graph) $G_{\rm sk}(P)= (V_{\rm sk}, E_{\rm sk})$ where $V_{\rm sk}$ is the set of centres of the tiles and where $E_{\rm sk}$ is the double line segment
path in the skeleton joining these vertices for adjacent tiles. The definition of $P$ determines a bijection $V(P_*) \to V_{\rm sk}$ and this extends to a (nonunique) piecewise linear bijection $\beta: G(P_*)\to G_{\rm sk}(P)$. This in turn induces a unique partition
$C(P)=\beta(C(P_*))$. The sets of this partition are the sets
\[
\beta(C_m) = \beta(A_1(B_{m_1})) \cap \dots  \cap \beta(A_r(B_{m_r}))
\]
and these are intersections of semi-open irregular bands $\beta(A_i(B_{m_i}))$ that are located between the skeletal curves of adjacent ribbons.

\begin{definition}\label{d:ribbonPhaseField} 
Let $P$ be the parallelogram tiling for a regular multigrid with ribbon partition $\{C_m: m \in \M\subset \bZ^r\}$ and let $\omega$ in $\bT^r$ be a multiphase. Then the \emph{ribbon phase field} $\phi_\omega :\bR^2 \to \bC$ is given  by
\[
\phi_\omega(x) = \omega^m, \quad \mbox{for}\mbox \quad x \in \beta(C_m).
\]
\end{definition}

The potential utility of ribbon phase fields is firstly that, in analogy with the crystallographic case, phase-periodic velocity fields can be defined by their restrictions when considering bar-joint frameworks $\G$ which are derived from $\G_P$ in some systematic manner. Secondly, taking a submultigrid, such as
$P_*'= \{A_1(\L'), \dots ,A'_r(\L')\}$ where
$\L'$ is the lattice $2\L = 2\bZ \times 2\bZ$, leads to a courser partition and associated phase fields which may be relevant to such derived frameworks. This is analogous to the crystallographic move of replacing a primitive periodicity basis by some other periodicity basis which enlarges the choice of building block.

Let $\rho_{i,k}, k\in \bZ,$ be a consecutive enumeration of the ribbons of a regular multigrid parallelogram tiling $P$  which are associated with the $i^{th}$ component grid.
Recall that a ribbon shear $u_{i,k}$ for $\rho_{i,k}$ is an infinitesimal flex that has zero velocity vectors  for the joints on one side of the ribbon and which has a common velocity vector on the other side. The latter vector  ${\bf t}_\rho$ 
is orthogonal to the internal edges of the ribbon. We can therefore choose ribbon shears
$w_{i,k}, k\in \bZ,$ in a consistent way, 
with common velocity ${\bf t}_i$, and with the support of $w_{i,k}$ contained in the support of $w_{i,k-1}$. Thus the differences $z_{i,k}= w_{i,k}-w_{i,k-1}$ are linearly localised infinitesimal flexes that are supported on the joints contained in the bands $\beta(A_i(B_{k}))$, for $k\in \bZ$. 

The infinitesimal flexes  $z_{i,k}$, for $ k \in \bZ,$ have disjoint supports and so for each $\lambda \in\bT$ the velocity vector
$
z_{i,\lambda}= \sum_k \lambda^k z_{i,k}
$
is an infinitesimal flex. It follows that $z_{i,\lambda}$ is a zero mode with phase-factor $\omega$ where $\omega_i=\lambda$ and $\omega_j=1, j\neq i$. 
Thus, the zero mode spectrum $\Omega(\G_P,\A)$ contains  the $r$-fold union
\[
(\bT\times \{1\}\times \dots  \times  \{1\}) \cup \dots \cup  
(\{1\}\times \dots \times \{1\}\times \bT).
\]
% does not form a free spanning set for the space of all infinitesimal flexes. Nevertheless 
It can be shown that each infinitesimal flex
$z_{i,k}$ can be represented as a pointwise convergent  infinite linear combination of the zero modes $z_{i,\lambda}$, with $1\leq i\leq r, k\in \bZ,$ and $\lambda = e^{2\pi i\gamma}, \gamma \in \bQ$. 
%This is a simple application of the principle expressed in the following elementary lemma. 
In the light of this, and in analogy with Example \ref{e:gridexample}, we expect that  $\Omega(\G_P,\A)$ is  equal to the $r$-fold union above.

%\begin{lemma}
%Let $\S$ be the set of functions $f_\gamma:\bZ\to \bT$ with $f(k) =e^{2\pi i\gamma k}, k\in \bZ$, where $\gamma\in \bQ$. Then for each $j\in \bZ$ the function $\delta_j$ with $\delta_j(k) = \delta_{jk},$ for $k\in \bZ$ can be written as an infinite sum $\sum_{\gamma \in \bQ}\alpha_\gamma f_\gamma$ which converges pointwise. 
%\end{lemma}

%-------------
% OLD Since the totality, $\Z$ say, of the zero modes $z_{i,\lambda}$ forms a free spanning set ***WRONG for the space of all infinitesimal flexes, we expect that  $\Omega(\G_P,\A)$  is equal to this $r$-fold union. Note that the infinite sum  of the  $z_{i,k}$, for fixed $i$, is the translational vector field with velocity vector  
%${\bf t}_i$. It follows that, for $r\geq3$, the set $\Z$ is not a free basis for $\F(\G_P)$. However, by Theorem \ref{t:shearbasis} it is a free spanning set.

% Note first that $z_{i,1}$ is an infinitesimal translation for each $i$ and that their span is 2-dimensional. Thus it suffices to show that each ribbon sheer $w_{i,k}$ can be recovered as an infinite linear combination of the $z_{i,\lambda}$ for $\lambda \in \bT$. 

\begin{rem}
Finally we remark that there are other forms of zero mode spectrum that can be defined, for aperiodic bar-joint frameworks $\G$, which are based in part on the presence  linearly localised flexes which can be phase-periodic in their localised directions. We consider this in detail elsewhere. The main idea for plane frameworks is to consider phase fields for band partitions in the ambient space of $\G$, associated with unrestricted parallelogram partitions, and to consider infinitesimal flexes with approximate forms of phase-periodicity. This leads to an associated zero mode spectrum,  consisting of lines of wave vectors ${\bf k}$ in a reciprocal space $\bR^2_{\bf k}$ relative to an  arbitrary reference basis for $\bR^2$.  For a regular multigrid parallelogram framework $\G_P$ this ``essential linear spectrum" is identified with the reciprocal figure of $RF(\G_P)$.
% in the reciprocal space for $\ul{a}$, determined by the ribbon figure $RF(\G_P)$. 
\end{rem}

\bibliographystyle{abbrv}
\def\lfhook#1{\setbox0=\hbox{#1}{\ooalign{\hidewidth
  \lower1.5ex\hbox{'}\hidewidth\crcr\unhbox0}}}

\end{document}